\documentclass[12pt]{amsart}
\usepackage{amstext,amsfonts,amssymb,amscd,amsbsy,amsmath,verbatim,fullpage}
\usepackage{ifthen}
\usepackage{color,tikz}
\usepackage{amsthm}
\usepackage{latexsym}
\usepackage[all]{xy}
\usepackage{enumerate}
\usepackage{pdfsync}

 \newtheorem{claim*}{Claim}

\theoremstyle{plain}
\newtheorem{theorem}{Theorem}

\newtheorem{corollary}[theorem]{Corollary}
\newtheorem{lemma}[theorem]{Lemma}

\newtheorem*{theoremnn}{Theorem}
\newtheorem{proposition.definition}[theorem]{Proposition/Definition}

\theoremstyle{definition}
\newtheorem{definition}[theorem]{Definition}

\newtheorem{remark}[theorem]{Remark}

\newtheorem{conjecture}[theorem]{Conjecture}

\newcommand{\im}{\operatorname{im}}
\newcommand{\cO}{{\mathcal O}}

\newcommand{\Sym}{\operatorname{Sym}} 

\newcommand{\lra}{\longrightarrow}

\newcommand{\noi}{\noindent}
\newcommand{\PP}{\mathbf{P}}

\newcommand{\ZZ}{\mathbf{Z}}

\newcommand{\OO}{\mathcal{O}}

\newcommand{\bull}{_{\bullet}}

\newcommand{\HH}[3]{H^{{#1}} \big( {#2} , {#3}
\big) }

\newcommand{\pr}{\prime}

\newcommand{\ol}[1]{\overline{#1}}

\newcommand{\bfkk}{\mathbf{k}}

\numberwithin{theorem}{section}

\newcommand{\Sdb}{S(b)^{(d)}}
\newcommand{\Mdb}{M(b)^{(d)}}

\begin{document}

\title{A Quick Proof of Nonvanishing for Asymptotic Syzygies}
  \author{Lawrence Ein}
  \address{Department of Mathematics, University Illinois at Chicago, 851 South Morgan St., Chicago, IL  60607}
 \thanks{Research of the first author partially supported by NSF grant DMS-1001336.}

\author{Daniel Erman}
\address{Department of Mathematics, University of Wisconsin, Madison, WI 53706}
\thanks{Research of the second author partially supported by NSF grant DMS-1302057}
 \author{Robert Lazarsfeld}
  \address{Department of Mathematics, Stony Brook University, Stony Brook, NY 11794}
 \thanks{Research of the third author partially supported by NSF grant DMS-1439285.}

\maketitle

\section*{Introduction}

The purpose of this note is to give a very quick new approach to   the main cases of the nonvanishing theorems of \cite{ASAV} concerning the asymptotic behavior of the syzygies of a projective variety as the positivity of the embedding line bundle grows. In particular, we present a surprisingly  elementary and concrete approach to the asymptotic nonvanishing of Veronese syzygies, and we obtain effective statements for arithmetically Cohen-Macaulay varieties. 

Let  $X$ be an irreducible projective variety of dimension $n$ over an algebraically closed field $\mathbf{k}$, and let  $L$ be a very ample divisor on $X$, defining an embedding
\[  X \ \subseteq \ \PP H^0(L) \ = \ \PP^{r}. \]
Write $S = \Sym H^0(L)$ for the homogeneous coordinate ring of $\PP^r$, and for a fixed divisor $B$ on $X$ consider the $S$-module
\[   M \ = \ M(B; L) \ = \ \bigoplus_m \, H^0(B + mL). \]
We are interested in the minimal graded free resolution $E\bull  = E\bull(B;L)$ of $M$ over $S$:
\[
 \xymatrix{
0 \ar[r]& E_r \ar[r]   & \ldots \ar[r] & E_1  \ar[r]  &  E_0 \ar[r]    &M \ar[r] &0,}
\]
with $E_p = \oplus S(-a_{p,j})$. Denote by \[ K_{p,q}(B;L) \ = \  K_{p,q}(X,B;L)\] the finite dimensional vector space of degree $p+q$ minimal generators of the $p^{\text{th}}$ module of syzygies of $M$, so that 
\[  E_p(B;L) \ = \ \bigoplus_q \, K_{p,q}(B;L) \otimes_k S(-p-q).  \]  (When $B = \OO_X$, we write simply $K_{p,q}(X;L)$ or $K_{p,q}(L)$ if no confusion seems likely.)  It is elementary that if $L$ is very positive compared to $B$ then non-zero syzygies can only appear in weights $0 \le q \le n+1$, and it turns out that the extremal cases $q =0$ and $q = n+1$ are easy to control. So the first interesting question is to fix $B$ and $1 \le q \le n$, and to ask which groups $K_{p,q}(B;L)$ are nonvanishing when  $L$ becomes very positive. The main result of \cite{ASAV} asserts in effect that -- contrary to what one might have expected by extrapolating from the case of curves -- these groups are eventually non-zero for almost all  values of $p \in [1,r]$. 

Perhaps the most natural instance of these matters occurs when $X = \PP^n$,  $B = \OO_{\PP^n}(b)$ and $L = L_d= \OO_{\PP^n}(d)$, so that one is looking at the syzygies of Veronese varieties. It was established in \cite{ASAV} that if one fixes $q \in [1,n]$ and $b \ge 0$, then for $d \gg 0$ one has
\begin{equation} \label{Vero.Non.Van.1}
 K_{p,q}(\PP^n, B; L_d) \ \ne \ 0 \end{equation}
for every value of $p$ satisfying
\small
\begin{equation} \label{Vero.Non.Van.Range}
\binom{d+q}{q} \, - \,  \binom{d-b-1}{q} \, - \, q   \ \le \  p \ \le \ \binom{d +n  }{n} \,  - \, \binom{d+n -q}{n-q} \, + \,  \binom{n+b}{n-q} \, - q-1. \end{equation}
\normalsize
For example, when $n = 2$ and $b = 0$, this asserts that 
\begin{equation} \label{OP.Non.Van} K_{p,2}(\PP^2; \OO_{\PP^2}(d) ) \ \ne \ 0  \ \ \text{for }\  3d - 2 \, \le p \, \le \binom{d+2}{2} -3,  \end{equation}
which was the main result of the interesting paper \cite{OP} of Ottaviani and Paoletti. The proof in \cite{ASAV} of the Veronese  nonvanishing theorem  involved a rather elaborate induction on $n$ to show that certain well-chosen secant planes to the Veronese variety force the presence of non-zero syzygies. For $b = 0$ the same  statement was obtained independently in characteristic zero by Weyman, who identified  certain representations of $\textnormal{SL}(n+1)$ that appear non-trivially in the $K_{p,q}$.  Some other work concerning Veronese syzygies appears in \cite{Rubei}, \cite{BCR}, \cite{BCR2},  and \cite{EGHP}, and a simplicial analogue of the results of \cite{ASAV} is given in \cite{CKW}.

The  goal of the present paper is to present a  much simpler and more elementary approach to the nonvanishing of Veronese syzygies, and to use this method to establish effective statements for arithmetically Cohen-Macaulay varieties. The idea is that one can reduce  the question to  elementary computations with monomials by modding out by a suitable regular sequence. In order to explain how this goes,  consider the problem of proving the first case of the Ottaviani-Paoletti statement \eqref{OP.Non.Van}, namely that if $d \ge 3$ then
\[
K_{3d-2,2}(\PP^2; \OO_{\PP^2}(d)) \ \ne \ 0. \tag{*}
\]
Writing $S_k$ for the degree $k$ piece of the polynomial ring $S = \mathbf{k}[x,y,z]$, it is well-known that the group in question can be computed as the cohomology at the middle term of the Koszul-type complex
\[
... \lra \Lambda^{3d-1} S_d \otimes S_{d} \lra  \Lambda^{3d-2} S_d \otimes S_{2d} \lra  \Lambda^{3d-3} S_d \otimes S_{3d} \lra ...
\]
The most naive approach to (*) would be  to write down explicitly a cocycle representing a non-zero element in $K_{3d-2,2}$, but we do not know how to do this.\footnote{The argument in \cite{OP} proceeds by using duality to reformulate the question as the nonvanishing of a $K_{p^\pr,0}$, where one can exhibit directly the required class.} On the other hand,  consider the ring
\[   \overline{S} \ = \ S / (x^d, y^d, z^d). \]
As $x^d, y^d, z^d$ form a regular sequence in $S$, the dimensions of the Koszul cohomolgy groups of $\overline{S}$ are the same as those of $S$, and hence the question is equivalent to proving the nonvanishing of the cohomology of
\[
... \lra \Lambda^{3d-1} \ol S_d \otimes \ol S_{d} \lra  \Lambda^{3d-2} \ol S_d \otimes \ol S_{2d} \lra  \Lambda^{3d-3} \ol S_d \otimes \ol S_{3d} \lra ...  \  . \tag{**}
\]
Now view  $\ol S$ as the ring spanned by monomials in which no variable appears with exponent $\ge d$, with multiplication governed by the vanishing of the $d^{\text{th}}$ powers of each variable. The plentiful presence of  zero-divisors in $\ol{S}$ means that one can write down by hand many monomial Koszul cycles: for instance if $m_1, \ldots, m_{3d-2}$ are monomials of degree $d$ each divisible by $x$ or $y$, then 
\[ c \ = \  m_1 \wedge \ldots \wedge m_{3d-2} \otimes x^{d-1}y^{d-1}z^2\]
gives  a cycle for the complex (**). Note next that $x^{d-1}y^{d-1}z^2$ has exactly $3d-2$ monomial divisors of degree $d$ with exponents $\le d-1$, viz:
\begin{gather*}
x^{d-1}y \, , \, x^{d-2}y^2\, , \, \ldots\, , \,, x^2y^{d-2}\, , \,x y^{d-1}     \\
x^{d-1}z\, , \, x^{d-2}y z \, , \, \ldots\, , \, xy^{d-2}z\, , \, y^{d-1}z   \\
x^{d-2}z^2 \, , \, x^{d-3}yz^2 \, , \, \ldots \, , \, xy^{d-3}z^2 \, , \, y^{d-2}z^2.
\end{gather*}
Taking these as the $m_i$, we claim that the resulting cycle $c$ represents a non-zero Koszul cohomolgy class. In fact, suppose that $c$  appears even as a term in the Koszul boundary of an element
\[  e \ = \ n_0  \wedge   n_1   \ \ldots \  \wedge n_{3d-2} \otimes g,\]
where the $n_i$ and $g$ are monomials of degree $d$. After re-indexing and introducing a sign we can suppose that 
\[   c \ = \ n_1 \wedge \ldots \wedge n_{3d-2} \otimes n_0g. \] Then the $\{ n_j \}$   with $j \ge 1$ must be a re-ordering of the monomials $\{ m_i \}$ dividing $x^{d-1}y^{d-1}z^2$. On the other hand $n_0g = x^{d-1}y^{d-1}z^2$, so $n_0$ is also such a divisor. Therefore $n_0$ coincides with one of $n_1, \ldots, n_{3d-2}$,  and hence $e  = 0$, a contradiction.  

We show that this sort of argument gives the nonvanishing of Veronese syzygies appearing in equation  \eqref{Vero.Non.Van.Range}, as well as a few further cases that were conjectured in \cite{ASAV}.  Moreover, we obtain a new statement that subsumes the previous statement and includes all values of $b,q,$ and $d$ (Theorem~\ref{PPn all b}). More interestingly, whereas the results of \cite{ASAV} for varieties other than $\PP^n$ were ineffective, we are able here to give effective statements for a large class of general varieties.

Specifically, consider an  arithmetically Cohen-Macaulay variety
$X \subseteq \PP^m$ of dimension $n$, and for $d > 0, b\ge 0$ write
\[   L_d \ = \ \OO_X(d) \ \ , \ \ B = \OO_X(b). \]  Put $ c(X)=  \min \big \{ k \, |\,  H^n(X, \OO_X(k-n)) = 0 \big \}$, the Castelnuovo-Mumford regularity of $\OO_X$, and write
\[ r_d \, =\, \dim H^0(X, \OO_X(d))  \ \ , \ \ r_d^\pr  \, = \, r_d - (\deg X)(n+1).  \]
We prove:
\begin{theoremnn}
Assume that $q \in [1, n-1]$, and fix $d \ge b + q + c(X) + 1$. Then
\[ K_{p,q}(X,B; L_d)\  \ne \ 0\] for every value of $p$ satisfying
\small
\[
\deg(X) (q + b + 1) \binom{d + q-1}{q-1} \ \le \ p \ \le \  {r}^\pr_d - \deg(X) (d-q-b) \binom{d + n-q-1 }{n-q-1}.
\]
\normalsize
\end{theoremnn}
\noi Analogous statements hold, with slightly different numbers, when $q =0$ and $q = n$; see Theorem~\ref{First.CM.Theorem} below.  We note that Zhou \cite{Zhou1} has given effective results for adjoint-type (and in particular, for very positive) line bundles $B$ on an arbitrary smooth complex projective variety. It would be interesting to know whether one could recover his statement by the present techniques: see Remark \ref{Adjoint.Type.Remark}. 

We wish to thank Xin Zhou for valuable discussions, and the referee for some suggestions which significantly streamlined the statement of Theorem~\ref{PPn all b}.

\setcounter{section}{1}

\numberwithin{equation}{section}
\section{Nonvanishing Results for $\PP^n$}

This section is devoted to the nonvanishing results for Veronese syzygies. 

Let $\bfkk$ be any field, and consider the polynomial ring $ S = \bfkk[x_0, \dots, x_n]$.  Given  $d\geq 1$ we denote by  $S^{(d)} \subseteq S$ the Veronese subring \[  S^{(d)} \ = \ \bigoplus_{j\in \ZZ} S_{jd} \ \subseteq \ S\]  of $S$.  For an $S$-module $M$, we write $\Mdb$ for the  $S^{(d)}$-module $\bigoplus_{j\in \ZZ} M_{b+jd}$.  Note that $\Mdb$ is also naturally a $\Sym(S_d)$-module.  We
denote by
\[  K_{p,q}(n,b;d) \ = \ K_{p,q}^{\Sym(S_d)}(\Sdb) \] 
 the Koszul cohomology group of $\Sdb$, where $\Sdb$ is considered as a $\Sym(S_d)$-module. Thus $K_{p,q}(n,b;d)$ is the cohomolgy of the  Koszul-type complex
\[  \ldots \lra \Lambda^{p+1}S_d \otimes S_{(q-1)d + b} \lra \Lambda^{p}S_d \otimes S_{qd + b} \lra \Lambda^{p-1}S_d \otimes S_{(q+1)d + b}\lra \ldots  \]
and \[ K_{p,q}(n,b;d) \ = \  K_{p,q}(\PP^n, \OO_{\PP^n}(b); \OO_{\PP^n}(d)). \]
 Since
\[
K_{p,q}(n,b;d) \ = \  K_{p,q+1}(n,b-d;d),
\]
we will always assume that $0\leq b \leq d-1$.

The following result is more precise than those in \cite{ASAV}, since in that paper, $b$ was always fixed and $d\geq n+1$.

\begin{theorem}\label{PPn all b}
Fix any $d$, any $b\in [0,d-1]$ and any $q\in [0, n+1-\frac{n+b}{d}]$.  Define $m$ and $r$ as the quotient and remainder of $qd+b$ by $d-1$. Then:
\[
K_{p,q}(n,b;d) \ \ne \ 0 \]
for all $p$ in the range
\small
\[
\binom{m+d}{m}-\binom{m+d-r-1}{m}-m \le p \le \binom{n+d}{n}+ \binom{n-m+r}{n-m}- \binom{n-m+d}{n-m}-m-1.
\]
\normalsize
\end{theorem}
\noi When $q\notin [0, n+1-\frac{n+b}{d}]$, then $K_{p,q}(n,b;d)$ is automatically zero; see Remark~\ref{rmk:q range}. On the other hand, if $d \ge n + b$, then the non-vanishing holds for all $0 \le q \le n$.
%
%

For the proofs, the idea is to mod out by a regular sequence to arrive at a situation where we can work by hand with monomials.  Specifically, by the technique of Artinian reduction, we can compute syzygies after modding out by a linear regular sequence.  Having fixed $d > 0$, we put 
\[ \overline{S} \ =_\text{def} \  S/(x_0^d, \dots, x_n^d).\]
Slightly abusively, we view $\ol{S}$ as the graded ring spanned by monomials in the $x_i$ in which no variable appears with exponent $\ge d$, with multiplication determined by the vanishing of the $d^{\text{th}}$ power of each variable.

  Since $x_0^d, \dots, x_n^d$ is a linear regular sequence in $\Sym S_d$,  modding out by these powers does not affect the Koszul cohomology groups. In other words:
\[
K_{p,q}^{\Sym(S_d)}(\Sdb) \ \cong \ K_{p,q}^{\Sym(\overline{S}_d)}\left(\Sdb\otimes_{\Sym(S_d)} \Sym(\overline{S}_d)\right)\ \cong \ K_{p,q}^{\Sym(\overline{S}_d)}(\overline{S}(b)^{(d)}).
\] 
It thus suffices to compute this last group, which is the homology at the middle of
\begin{equation}\label{Monomial.Complex}
 \xymatrix{
 \bigwedge^{p+1} \overline{S}_d \otimes  \overline{S}_{(q-1)d+b}\ar[r]^-{\partial_{p+1}}& \bigwedge^{p}  \overline{S}_d\otimes  \overline{S}_{qd+b}\ar[r]^-{\partial_p}& \bigwedge^{p-1}  \overline{S}_d\otimes  \overline{S}_{(q+1)d+b}.
 }
\end{equation}
In particular, $K_{p,q}(n,b;d) \ne 0$ if and only if this complex has non-trivial homology, and we are therefore reduced to studying  cycles and boundaries  in  \eqref{Monomial.Complex}.

We start with some notation that will prove useful. Fix a finite set of elements $P \subseteq \ol{S}$ (which in practice will be a collection of monomials).
\begin{definition}
 We write 
$\zeta \in \bigwedge P$ (or $\zeta \in \bigwedge^s P$) if \[
\zeta\ = \ m_1\wedge \dots \wedge m_s\]with $m_i\in  P$  for all $i$.  We write $\zeta = \det   P$ if $\zeta$ is the wedge product of all elements in $  P$ (in some fixed order).  We say that a wedge product $m_1\wedge \dots \wedge m_s$ is a monomial if each $m_i$ is a monomial. \end{definition}

The following lemma guarantees the existence of many non-zero monomial classes in the cohomology of \eqref{Monomial.Complex}. It systematizes the computations worked out for a special case in the Introduction.
\begin{lemma}\label{lem:cycles and boundaries}
Fix a nonzero monomial $f\in  \overline{S}_{qd+b}$, and denote by
\[   Z_f \ , \ D_f \ \subseteq \ \ol{S}_d \]
respectively the set of degree $d$ monomials that annihilate or divide $f$.  
\begin{itemize}
	\item[(i).]   If $\zeta \in \bigwedge^p Z_f$,  then
$
\zeta \otimes f \in \ker \partial_p.
$
\vskip 8pt
	\item[(ii).]  Let $\zeta \in \bigwedge^s  \overline{S}_d$ be any monomial such that such that $\det  D_f \wedge \zeta \otimes f$ is nonzero.  Then
\[
(\det    D_f \wedge \zeta) \otimes f \  \notin \ \im \partial_{(|  D_f|+s)}.
\]
\end{itemize}
\end{lemma}
\begin{proof}
For (i), write $\zeta=m_1\wedge \dots \wedge m_s$ with $m_i\in  Z_f$.  Since $m_{i}f=0 \in \ol{S} $ for all $i=1,\dots, s$, the assertion is immediate. Turning to (ii), 
assume that
\[  \partial \big( \sum \xi_j\otimes g_j\big) \ =  \ \big(\det   D_f\wedge \zeta \big) \otimes f \]  Then there must be some index $j$ and some monomial appearing in  $\xi_j \otimes g_j$ that maps to the monomial $\big(\det    D_f\wedge \zeta\big) \otimes f$.  In particular, $\xi_j \otimes g_j$ must contain  a non-zero monomial of the form $(m\wedge \det   D_f \wedge \zeta) \otimes g$ where $mg=f$.  But then $m\in   D_f$ and hence 
$m\wedge \det  D_f=0$, a contradiction.
\end{proof}

\begin{corollary} \label{Kpq.nonzero.corollary}
Given $q,d$ and $b$,  let $f\in  \overline{S}_{qd+b}$ be a monomial such that $D_f\subseteq  Z_f$.  Then any non-zero monomial of the form 
\[ \big(\det  D_f\wedge \zeta \big) \otimes f, \] where $\zeta\in \bigwedge Z_f$, represents a nonzero element of the cohomology of \eqref{Monomial.Complex}.
In particular, \[ K_{p,q}(n,b;d) \ \ne \ 0\] for every $p$ satisfying
\[  |   D_f | \ \le  \ p \ \le \ |Z_f|. \ \ \ \ \qed\]
\end{corollary}

\begin{remark}
The Koszul classes just constructed are linearly independent. In fact, keeping the notation of the corollary, and with an appropriate degree twist, there is a natural map from the Koszul complex on the linear forms in $Z_f$ to the minimal free resolution of $ \overline{S}(b)^{(d)}$ over $\Sym \ol{S}_d$ given by sending $1\mapsto f$.  This induced map yields an inclusion of the Koszul subcomplex on the linear forms 
\[  Z_f \setminus D_f \ \subseteq \Sym(\ol{S_d}) \] spanning homological degrees $p=|  D_f|,|  D_f|+1,  \dots, |  Z_f|$.  
In Conjecture~B from \cite{ARBT}, we conjectured that each row of the Betti table of a high degree Veronese looks roughly like the Betti table of a Koszul complex.  Although this result has a similar flavor, the lower bound on the size of the Koszul cohomology groups constructed via this method is far too small to verify the conjecture of \cite{ARBT}.
\end{remark}

Theorems~\ref{PPn all b} now follows from Corollary \ref{Kpq.nonzero.corollary} by choosing a convenient monomial $f$ and computing the number of elements in the resulting sets $Z_f$ and $D_f$.
\begin{proof}[Proof of Theorem~\ref{PPn all b}] Put
\[   s_d \ = \ \dim \ol{S}_d \ = \ \binom{n+d}{d} - (n+1). \]
Let $f$ be the ``leftmost" monomial of $ \overline{S}$ having degree $dq+b$; by our definition of $m$ and $r$ this is the monomial of the form:
\[
f \ = \ x_0^{d-1}\cdot x_{1}^{d-1}\cdot \ldots\cdot x_{m-1}^{d-1}\cdot x_m^{r}.
\]
In order to establish the theorem, it suffices to prove three assertions:
\begin{itemize} 
	\item[(i).] $s_d- |  Z_f| \, =\,   \binom{d+n-m}{d}  - \binom{n+r-m}{r} - n + m -2.$
	\vskip 5pt
	\item[(ii).]  $| D_f|\, =\, \binom{m+d}{m}-\binom{m+d-r-1}{m}-m.$
	\vskip 5pt
	\item[(iii).]  $D_f\subseteq  Z_f$.
\end{itemize}
For (i), observe that $Z_f = (0:_{\overline{S}} f)_d$ contains all monomials of degree $d$ that are divisible by any of $x_0, \dots, x_{m-1}$ as well as those divisible by $x_m^{r}$.  Hence among the $s_d$ monomials in $ \overline{S}_d$, the ones not lying in $ Z_f$ are  the monomials  of degree $d$ appearing in the quotient
\[
\overline{S}/(x_0,\dots, x_{m-1},x_m^{d-r}).
\]
We can compute this via the resolution:
\[
\cdots
\longrightarrow
\frac{\overline{S}(-d)}{(x_0,\dots, x_{m-1})}\overset{\cdot x_m^{r}}{\longrightarrow}
\frac{\overline{S}(-d+r)}{(x_0,\dots, x_{m-1})}\overset{\cdot x_m^{d-r}}{\longrightarrow}
\frac{ \overline{S}}{(x_0,\dots, x_{m-1})}\longrightarrow  \frac{\overline{S}}{(x_0,\dots, x_{m-1},x_m^{d-r})}.
\]
Therefore
\begin{align*}
s_d - | Z_f| &=\dim_k \left(  \overline{S}/(x_0, \dots, x_{m-1},x_m^{d-r})\right)_d\\
&=\dim ( \overline{S}/(x_0, \dots, x_{m-1}))_d  - \dim ( \overline{S}/(x_0, \dots, x_{m-1}))_{r} + \dim ( \overline{S}/(x_0, \dots, x_{m-1}))_0\\
&=\left( \binom{d+n-m}{d} -n+m-1 \right)- \binom{n+r-m}{r}+1.
\end{align*}
For (ii), note that $  D_f$ can be identified with  the degree $d$ monomials of $ \overline{S}/(x_m^{r+1}, x_{m+1}, \dots, x_n)$.  A similar computation yields
\begin{align*}
| D_f| \ &= \ \dim \left(  \overline{S}/(x_m^{r+1}, x_{m+1}, \dots, x_n)\right)_d \\
&=\ \dim \big ( \overline{S}/(x_{m+1}, \dots, x_{n})\big)_d  - \dim \big( \overline{S}/(x_{m+1}, \dots, x_{n})\big)_{d-r-1} + \dim \big( \overline{S}/(x_{m+1}, \dots, x_{n})\big)_0\\
&=\ \left( \binom{m+d}{d} - m-1 \right)- \binom{d-r+m-1}{m}+1.
\end{align*}
Finally, since the exponent of $x_m$ in $f$ is $r\le d-1$, it follows that any element of $D_f$ is divisible at least by one of $x_0, \dots, x_{q-1}$, and hence any such element annihilates $f$. 
\end{proof}
%
%
\begin{remark}\label{rmk:q range}
If $q<0$ then since $b\geq 0$, we clearly have $K_{p,q}(n,b;d)=0$ for all $p$.  
If, $q>n+1-\frac{n+b}{d}$ then we claim that we also get vanishing for all $p$.  We define $q':=n+1-q$ and note that we the above inequality on $q$ is equivalent to having $q'd<n+b$.  We then use duality to compute:
\[
K_{p,q}(n,b;d)^* = K_{r_d-n-p,q'}(n,-n-1-b;d).
\]
Since $\cO(-n-1-b+q'd)$ has no sections when $q'd<n+1+b$, our assumptions imply that this group equals $0$ for all $r_d-n-p\geq 0$ and hence for all $p\geq 0$.
\end{remark}

\begin{remark}
Zhou \cite{Zhou2} has recently established some  results about the asymptotic distribution of torus weights appearing in the $K_{p,q}$ of toric varieties. It would be interesting to know if the present arguments can be used to give more refined information in the case $X = \PP^n$. 
\end{remark}

%

\begin{remark}
It is conjectured in \cite[Conjecture 7.5]{ASAV} that for $d\geq n+1$, the assertion of Theorem \ref{PPn all b} is optimal in the sense that the $K_{p,q}$ in question vanish outside the stated range, and we conjecture that the more general bounds in Theorem \ref{PPn all b} are optimal as well. 

For instance, in the case $d=2$ and $b=0$, the full resolution is known in characteristic $0$ by work of J\'{o}zefiak, Pragacz and Weyman in \cite{JPW}.  Their theorem shows that, as long as $n+1 \geq 2q$, $K_{p,q}(n,0;2)=0$ starting with $p=2q^2-q$, and this value lines up with the lower bound in Theorem~\ref{PPn all b}.

It would be exceedingly interesting to know whether one can use the approach introduced here to make progress on this conjecture, at least in the case $d \gg 0$ as in \cite[Problem 7.7]{ASAV} . Unfortunately it seems that one can't work only with monomials -- it's possible for instance that a monomial Koszul cocyle appears as the boundary of non-monomial elements. It is tempting to wonder whether there are Gr\"obner-like techniques that could be used to study the issue systematically. We note that Raicu \cite{Raicu} has reduced the general vanishing conjecture   \cite[Conjecture 7.1]{ASAV} to the case of Veronese syzygies.
\end{remark}


\section{Nonvanishing for arithmetically Cohen-Macaulay schemes}
In this section we extend the results of the previous section to the setting of arithmetically Cohen-Macaulay schemes.

Consider an arithmetically Cohen-Macaulay scheme $X\subseteq \PP^{m}$ of dimension $n$  over the field $\bfkk$, and let 
\[ R \ = \ \oplus \, H^0(X, \OO_X(k))\]   be the  homogeneous coordinate ring of $X$. Setting $L_d = \OO_X(d)$ and $B = \OO_X(b)$, we are interested in the syzygies
\[  K_{p,q}(X, B; L_d) \ = \  K_{p,q}(R(b)^{(d)})\] of $B$ with respect to $L_d$ for $d \gg 0$. Put 
\[  c\  = \  c(X) \ = \  \min \big \{ k \, |\,  H^n(X, \OO_X(k-n)) \, = \, 0 \big \},\]
and write 
\[ r_d \, =\, \dim H^0(X, \OO_X(d))\, = \, \dim R_d \ \ , \ \ r_d^\pr  \, = \, r_d - (\deg X)(n+1).  \]

Our first result holds when $d \ge b +q + c + 1$.
\begin{theorem} \label{First.CM.Theorem}  Fix $b \in [0, d-q-1-c]$. 
\begin{enumerate}
\item[$($i$)$.] If $q \in [1, n-1]$, then
$ K_{p,q}(X,B; L_d) \ne 0$ for
\small
\[
(\deg X) (q + b + 1) \binom{d+q-1}{q-1} \ \le \ p \ \le \  {r}^\pr_d - (\deg X) (d-q-b) \binom{d+n-q-1}{n-q-1}.
\]
\normalsize
\vskip 4pt
\item[$($ii$)$.] When $q = n$, one has  $K_{p,n}(X, B; L_d) \ne 0$ when
\[ (\deg X)(n+b+1) \binom{d+n-1}{n-1} \ \le \ p \ \le \ r_d^\pr  - \deg X. \]
\vskip 4pt
\item[$($iii$)$.]  When $q = 0$ one has $K_{p,0}(X, B;L_d) \ne 0$ when
\[  0 \ \le \ p \ \le \ r_d^\pr - (d-b)\binom{n-1+d}{n-1}. \]
\end{enumerate}
\end{theorem}
 \noi A somewhat more complicated but sharper statement appears in Remark \ref{rmk:improved acm bound} below.

 \begin{remark}
If we fix $b$ and $q$, we can interpret these bounds as asymptotic statements as $d\to \infty$.  Under these assumptions, we are saying that $K_{p,q}(X,B; L_d) \ne 0$ for all $p$ in the range
 \[
 \frac{\deg(X)(q+b+1)}{(q-1)!}d^{q-1} + O(d^{q-2})\  \leq \ p \  \leq \ r'_d -\left(  \frac{\deg(X)}{(n-q-1)!}d^{n-q} + O(d^{n-q-1})\right) 
 \]
Conjecture~7.1 in \cite{ASAV} states that one should have $K_{p,q} = 0$ for $p \le O(d^{q-1})$; it would be interesting to understand the optimal leading coefficients as well.  In the ACM case this implies that $K_{p,q} = 0$ also for $p > r_d - O(d^{n-q})$, but in the non-ACM case the groups in question can be nonvanishing for $p \approx r_d$ \cite[Remark 5.3]{ASAV}. 
\end{remark}

For the proofs of the theorem, let $I_X\subseteq \bfkk[x_0,\dots,x_m]$ be the defining ideal of $X$, so that $R=\bfkk[x_0,\dots,x_m]/I_X$.  The statement is independent of the ground field, so we may assume that $\bfkk$ is  infinite.  Then, after a general change of coordinates, we may assume that $x_0, \dots, x_n$ form a system of parameters for $R$.  To help clarify  the following arguments, we will relabel the variables $x_{n+1},\dots,x_m$ as $y_{n+1},\dots,y_m$.

Let $S=\bfkk[x_0,\dots,x_n]\subseteq R$, which is a Noether normalization since  $x_0, \dots, x_n$ is a system of parameters.  As $R$ is Cohen-Macaulay of dimension $n+1$, it follows that it is a maximal Cohen-Macaulay $S$-module, and hence a free $S$-module.  We may choose a set $\Lambda$  of monomials of the form $y^\beta\in R$ which form a basis for $R$ as an $S$-module, so that
\[
R \ = \ \bigoplus_{y^\beta \in \Lambda} \, S \cdot y^\beta.
\] We assume that $1\in \Lambda$. Thus $\deg(X) = \# \Lambda$ and we observe that $ c(X) = \max \{ \deg y^\beta  \}$.

Fix $q \in[0,n]$, $d > 0$  and $b\ge 0$. Set 
\[ \overline{R} \ = \ R/(x_0^d,\dots,x_n^d), \] and define $\overline{S}$ as in the previous section. Thus $\overline{R} = R \otimes_S \overline{S}$, and $\ol{R}$ is  a free $\overline{S}$-module with basis $\Lambda$.
 Since $R$ is Cohen-Macaulay, we have 
\[\dim  K_{p,q}(R(b)^{(d)})\ =\ \dim K_{p,q}(\overline{R}(b)^{(d)})\] for all $p$ and $q$, where the group on the right is computed as the cohomology of the complex
\begin{equation}\label{Coh.Mac.Mon.Complex}
 \xymatrix{
 \bigwedge^{p+1} \overline{R}_d \otimes  \overline{R}_{(q-1)d+b}\ar[r]^-{\partial}& \bigwedge^{p}  \overline{R}_d\otimes  \overline{R}_{qd+b}\ar[r]^-{\partial}& \bigwedge^{p-1}  \overline{R}_d\otimes  \overline{R}_{(q+1)d+b}.
 }
\end{equation}
In the natural way, we can speak of monomials in $\ol{R}$: these are (the images in $\ol{R}$ of) elements of the form $x^\alpha y^\beta$ where $y^\beta \in \Lambda$, and the degree of such a monomial is $|\alpha| + |\beta|$. Given a monomial $f \in \ol{S}$, we denote by 
\[   Z_f \  , \    E_f \ \subseteq \ \ol{R}_d \]
respectively the set of degree $d$ monomials that annihilate $f$ and the collection of degree $d$ monomials of the form $x^\alpha y^\beta$ where $x^\alpha$ divides $f$ and $y^\beta \in \Lambda$. 

We start with an analogue of Lemma \ref{lem:cycles and boundaries}. 
\begin{lemma} \label{CM.Non-Van.Cycle.Lemma} Let \[ f \in \ol{S}_{qd + b} \ \subseteq  \ \ol{R}_{qd + b}\] be a monomial such that $E_f \subseteq Z_f$. Then any non-zero monomial element
\[   m = \big( \det E_f \wedge \zeta \big) \otimes f \] with $\zeta \in \bigwedge Z_f $
represents a non-zero Koszul cohomology class. In particular
\[ K_{p,q}\big (X, \OO_X(b) ; \OO_X(d) \big) \ \ne \ 0 \]
for every $p$ with
\[   | E_f | \ \le \ p \ \le \ |Z_f|. \]
\end{lemma}

\begin{proof}
Since $E_f \subseteq Z_f$, $m$ is evidently a Koszul cycle. It remains to prove that it is not cohomologous to zero. In fact, we'll show that $m$ cannot occur as a monomial appearing in the expansion (with respect to the chosen basis of $\ol{R}$) of $\partial(\xi \otimes g)$ for any monomials $\xi \in \Lambda^{p+1} \ol{R}_d$ and $g \in \ol{R}_{(q-1)d + b}$. Suppose to the contrary that $m$ appears as a term in $\partial( \xi_0 \wedge \ldots \wedge \xi_p \otimes g)$. Then after possibly reindexing and introducing a sign, we can suppose
\[  \xi_1 \wedge \ldots \wedge \xi_p \ = \ \det(E_f)\wedge \zeta, \] and that  $f$ appears as a term in the expansion of $\xi_0 g$ in terms of the basis $\Lambda$. Suppose that
\[  \xi_0 \ = \ x^\alpha y^\beta \ \ , \ \ g = x^\gamma y^\delta \]
where $y^\beta, y^\delta \in \Lambda$. Then in $\ol{R}$ we can rewrite
\[  y^{\beta + \delta } \ = \ h_0 \cdot 1 \, + \, \sum_{y^\lambda \in \Lambda} \, h_{\lambda} \cdot y^\lambda \]
where $ h_{\lambda}  \in   \ol{S}_{|\beta| +|\delta| - |\lambda|}$. Therefore $f = x^{\alpha+\gamma}h_0$, and consequently $x^\alpha y^\beta \in E_f$. In particular $\xi_0$ also appears as one of $\xi_1, \ldots, \xi_p$, and hence $m = 0$.
\end{proof}

We now turn to the
\begin{proof}[Proof of Theorem \ref{First.CM.Theorem}]
As before, we take $f$ to be the be the leftmost nonzero monomial of $\overline{S}$ of degree $dq+b$:
\[    f \ = \ x_0^{d-1}\cdot x_{1}^{d-1}\cdot \ldots \cdot x_{q-1}^{d-1}\cdot x_q^{q+b}.\]
We claim first of all that $E_f \subseteq Z_f$ provided that $d \ge b +q + c + 1$. In fact, suppose that
\[  w \ = \ x_0^{a_0} \cdot \ldots \cdot x_q^{a_q}\cdot y^\beta \ \in \ E_f. \]
Then $a_q \le q+b$, and hence
\begin{align*} a_0 + \ldots + a_{q-1} \ &= \ d - a_q - |\beta| \\ &\ge \ d - (q+b) - c \\ &> 0. \end{align*}
Therefore at least one of $a_0, \ldots, a_{q-1}$ is strictly positive, and consequently $w \in Z_f$. 

In order to apply Lemma \ref{CM.Non-Van.Cycle.Lemma}, it remains to estimate the sizes of $E_f$ and $Z_f$. Writing $\ol{r}_d = \dim \ol{R}_d$, we start by giving an upper bound on $\ol{r}_d - | Z_f|$. Assume first that $q \in [1, n-1]$, and consider a monomial $x^{\alpha} = x_0^{a_0} \cdot \ldots \cdot x_n^{a_n}$.  Then 
a degree $d$ monomial $x^\alpha y^\beta$ lies in the complement of $Z_f$ if and only if
\[ a_0 \,  = \, \ldots \, = \, a_{q-1} \, = \, 0 \ \ , \ \ a_q \le d - b -q - 1. \]
The number of  possibilities for $x^\alpha$ is (rather wastefully) bounded above simply by the number 
of degree $d$ monomials in ${\bf k}[x_{q+1},\dots,x_{n}]$, times the number of choices for $a_q$, times the number of choices for $y^\beta$.
Since $|\Lambda| = \deg X $, this leads to the lower bound
\[    \ol{r}_d - (\deg X) (d-q-b) \binom{d + n - q -1}{n-q-1} \ \le \ |Z_f|.\] Turning to an upper bound on $|E_f|$, observe that $x^\alpha y^\beta \in E_f$ if and only if
\[   a_0 , \ldots, a_{q-1} \, \le \, d-1 \ \ , \ \ a_{q} \le q+b \ \ \text{ and } \ \ a_{q+1}=\dots=a_n=0 \]
We can bound this (again wastefully) by the number of monomials of degree $d$ in ${\bf k}[x_0,\dots,x_{q-1}]$, times the number of choices for $a_q$, times the number of choices for $y^\beta$.  This leads to:
\[  (\deg X) (q+b+1) \binom{q-1 + d}{q-1} \ \ge \ |E_f|. 
\]

\noi So to obtain assertion (i) of Theorem  \ref{First.CM.Theorem}, it remains only to observe that
\begin{align*}  \ol{r}_d \ &= \ \sum_{y^\beta \in \Lambda} \dim \ol{S}_{d - |\beta|}\\ &\ge \ \sum_{y^\beta \in \Lambda}  \left ( \dim S_{d - |\beta|} - (n+1)\right ) \\ &= \ \dim R_d \ - \ |\Lambda| (n+1) \\ &= \ r_d^\pr.
\end{align*}

When $q = n$ we get the same bound on $|E_f|$, but now we find that 
\[  \ol{r}_d \, - \, ( \deg X ) \ \le \ |Z_f|, \]
and this yields statement (ii) of the Theorem. Finally, when $q = 0$ we get the same lower bound on $|Z_f|$ as above, and we can obtain nonvanishing starting with $p=0$.
\end{proof}

\begin{remark}\label{rmk:improved acm bound} By separating the estimates into two terms depending on whether $\beta$ is equal to zero or not, one gets a slightly better upper bound on the size of $E_f$, when $q \in [1, n-1]$:
\[  (\deg X -1) (q+b+1) \binom{q-1 + d-1}{q-1} + \binom{q+d}{q}- \binom{d-b-1}{q}-q\ \ge \ |E_f|. 
\]
In particular, this reduces to the statements obtained for $\PP^n$ in the previous sections. 
\end{remark}

\begin{remark}\label{rmk:ACM big d} 
By defining $m$ and $r$ as the quotient and remainder of $dq+b$ by $d-1$, one can use an argument to the proof of Theorem~\ref{PPn all b} to extend this to some additional values of $q,d,$ and $b$.  However, we felt the asymptotic behavior was more clear when phrased in terms of $q$ and $b$ instead of $r$ and $m$.
\end{remark}

\vskip 10pt
\begin{remark}
The bounds for $|E_f|$ and $\ol{r}_d - |Z_f|$ appearing in the proof of Theorem \ref{First.CM.Theorem} 
could be improved by a more precise count of the relevant possibilities, in particular taking into account the degrees of the $y^\beta$.  This amounts to computations involving the numerator of the Hilbert series of $R$ (i.e. the Hilbert function of the Artinian reduction $\overline{R}$), and confronted with a specific example, it is often quite easy to use directly the method of the proof to get stronger statements. For example, let $X \subseteq \PP^5$ be the hypersurface 
\[ x_0^3 + \ldots + x_5^3 \ = \ 0. \]    Then $\Lambda = \{1,x_5,x_5^2\}$, so $c = 2$. We take $(q,b,d)=(3,0,8)$ and 
\[ f \ = \ x_0^{7}x_1^{7}x_2^{7}x_3^3. \]
Then $R={\bf k}[x_0,\dots,x_5]/(x_0^3 + \ldots + x_5^3)$ and $\overline{R}=R/(x_0^8,\dots,x_4^8)$.
The bounds from Theorem~\ref{First.CM.Theorem} and Remark 2.5 yield the nonvanishing result $K_{p,3}\big(X; \OO_X(8)\big) \ \ne  \ 0$ for $p$ between $540$ and $1005$.

However, if we follow the method of the proof, we can compute the size of $E_f$ directly.  Let $A:={\bf k}[x_0,\dots,x_{q}]/(x_0^d,\dots,x_{q-1}^d,x_q^{q+b+1})={\bf k}[x_0,\dots,x_3]/(x_0^8,x_1^8,x_2^8,x_3^4)$.  Then
\[
\sum_{y^\beta\in \Lambda} \dim A_{d-\deg y^\beta}=\dim A_8+\dim A_7+\dim A_6=301.
\]
A similar computation shows that there are $14$ monomials in the complement of $Z_f$ and so $|Z_f|=1030-14=1016$, and the nonvanishing statement can be extended to all values of $p$ between $301$ and $1016$.

\end{remark}

\begin{remark} \label{Adjoint.Type.Remark}
Let $X \subseteq \PP^m$ be an arbitrary variety of dimension $n$, and suppose that $B$ is a line or vector bundle on $X$ with the property that
\[ \HH{i}{X}{B\otimes \OO_X(k)}  \ = \ 0 
  \]
  for all $k \in \ZZ$ and $0 < i < n$: in other words, $M = \oplus \HH{0}{X}{B\otimes \OO_X(k)} $ is a Cohen-Macaulay module over the homogeneous coordinate ring of $\PP^m$. Replacing $B$ by a twist, one can assume without loss of generality that $M_{-1} =0$  but $M_0 \ne 0$. Then one can use the methods of this section to obtain effective nonvanishing statements for the syzygies $K_{p,q}(X, B; \OO_X(d))$. In fact, the hypotheses on  $M$ imply that it has a generator in degree $0$, and then in the arguments above one can replace $R$ by $M$. We leave details to the interested reader. It would be interesting to compare the resulting statements with the results \cite{Zhou1} of Zhou which fall under this rubric. 
\end{remark}

Finally, we expect that nonvanishing statements similar to Theorem~\ref{First.CM.Theorem} hold for any finitely generated, graded $\mathbf{k}$-algebra $R$.  More precisely, we conjecture the following analogue of part (i) of Theorem~\ref{First.CM.Theorem}.
\begin{conjecture} Fix $b$ and $R$ and $q\in [1,n]$ where $n:=\dim R-1$. Then there exist constants $c$ and $C$ such that if $d \gg 0 $ then
\[  K_{p,q}(R(b)^{(d)})\ \ne  \ 0 \ \ \text{for all} \ \ cd^{q-1} \, \leq \,  p\,  \leq r_d - Cd^{n-q}\] and for all $d\gg 0$.
\end{conjecture}
 \noi We expect similar analogues of parts (ii) and (iii) of Theorem~\ref{First.CM.Theorem}, as well as analogues of the cases where $b$ is close to $d$, as in Remark~\ref{rmk:ACM big d}.

\end{document}